\documentclass{amsart}

\usepackage{amssymb}
\usepackage{textcomp}
\usepackage{bbm}

\newtheorem{thm}{Theorem}[section]
\newtheorem{cor}[thm]{Corollary}
\newtheorem{lem}[thm]{Lemma}
\newtheorem{prop}[thm]{Proposition}
\theoremstyle{definition}
\newtheorem{defn}[thm]{Definition}
\newtheorem{ex}[thm]{Example}
\newtheorem{rem}[thm]{Remark}

\numberwithin{equation}{section}

\usepackage{setspace}
\usepackage{color}

\definecolor{grey}{gray}{.3}

\newcommand{\bdefn}{\begin{defn}}
\newcommand{\edefn}{\end{defn}}
\newcommand{\bdefnr}{\begin{defnr}}
\newcommand{\edefnr}{\end{defnr}}
\newcommand{\benum}{\begin{enumerate}}
\newcommand{\benuma}{\begin{enumerate}}
\newcommand{\eenum}{\end{enumerate}}
\newcommand{\eenuma}{\end{enumerate}}
\newcommand{\bthmn}{\begin{thmn}}
\newcommand{\ethmn}{\end{thmn}}
\newcommand{\bthm}{\begin{thm}}
\newcommand{\ethm}{\end{thm}}
\newcommand{\bnota}{\begin{nota}}
\newcommand{\enota}{\end{nota}}
\newcommand{\bproof}{\begin{proof}}
\newcommand{\eproof}{\end{proof}}
\newcommand{\bprop}{\begin{prop}}
\newcommand{\eprop}{\end{prop}}
\newcommand{\bcor}{\begin{cor}}
\newcommand{\ecor}{\end{cor}}
\newcommand{\blem}{\begin{lem}}
\newcommand{\elem}{\end{lem}}
\newcommand{\brem}{\begin{rem}}
\newcommand{\erem}{\end{rem}}
\newcommand{\bex}{\begin{ex}}
\newcommand{\eex}{\end{ex}}
\newcommand{\boq}{\begin{openq}}
\newcommand{\eoq}{\end{openq}}
\newcommand{\btab}{\begin{tabular}}
\newcommand{\etab}{\end{tabular}}
\newcommand{\bctab}{\begin{center}\begin{tabular}}
\newcommand{\ectab}{\end{tabular}\end{center}}

\newcommand{\ba}{\begin{array}}
\newcommand{\ea}{\end{array}}
\newcommand{\bea}{\begin{eqnarray}}
\newcommand{\eea}{\end{eqnarray}}
\newcommand{\bean}{\begin{eqnarray*}}
\newcommand{\eean}{\end{eqnarray*}}

\newcommand{\cbb}{\mathbb{C}}
\newcommand{\rbb}{\mathbb{R}}

\newcommand{\nbb}{\mathbb{N}}

\newcommand{\acal}{\mathcal{A}}

\newcommand{\ccal}{\mathcal{C}}

\newcommand{\bcal}{\mathcal{B}}

\newcommand{\ph}{pluri\-har\-monic}

\newcommand{\psh}{pluri\-sub\-har\-monic}

\newcommand{\subph}{sub\-pluri\-har\-monic}
\newcommand{\Subph}{Sub\-pluri\-har\-monic}
\newcommand{\qpsh}{$q$-pluri\-sub\-har\-monic}
\newcommand{\qpshy}{$q$-pluri\-sub\-har\-monicity}
\newcommand{\qpshies}{$q$-pluri\-sub\-har\-monicities}

\newcommand{\sqpsh}{strictly $q$-pluri\-sub\-har\-monic}

\newcommand{\PSH}{\mathit{PSH}}

\newcommand{\psc}{pseudo\-convex}

\newcommand{\qpsc}{$q$-pseudo\-convex}

\newcommand{\nbh}{neighborhood}

\newcommand{\usc}{upper semi-con\-ti\-nuous}

\newcommand{\fcts}{functions}

\renewcommand{\and}{\mathrm{and}}

\renewcommand{\Re}{\mathrm{Re}}
\renewcommand{\Im}{\mathrm{Im}}

\title[On convex hulls and pseudoconvex domains]
{On convex hulls and pseudoconvex domains generated 
by $q$-plurisubharmonic functions,\\ part III}

\author{T. Pawlaschyk}
\address{University of Wuppertal, School of Mathematics and Natural Sciences, Gaussstrasse 20, 42119 Wuppertal, Germany}
\email{pawlaschyk@math.uni-wuppertal.de}

\author{E. S. Zeron}
\address{Departamento de Matem\'aticas, CINVESTAV del IPN, 
Apartado Postal 14-740, Ciudad de M\'exico, 07000, Mexico}
\email{eszeron@math.cinvestav.edu.mx}

\thanks{Research supported by the Deutscher Akademischer Austauschdienst 
(DAAD) and Conacyt M\'exico under the PPP Proalmex Project No. 51240052. 
The first author was supported by the Deutsche Forschungsgemeinschaft 
(DFG) under the grant SH 456/1-1, {\it Pluripotential Theory, Hulls and 
Foliations}; and the second author was supported by Cinvestav del IPN in 
M\'exico while he was in a sabbatical leave in CRM de l'Universit\'e de 
Montr\'eal in Canada.\\
The original manuscript was submited for its publication to 
the {\bf Journal of Mathematical Analysis and Applications}}



\date{\today}

\begin{document}
\begin{abstract} 
We characterise in this work the \qpsh\ functions in terms of the 
theory of viscosity solutions. We show that an \usc\ function is 
\qpsh\ if and only if its complex Hessian has at most $q$ strictly 
negative eigenvalues in the viscosity sense. This characterisation is 
then used to prove that the sup-convolution of a (strictly) \qpsh\ 
function is again (strictly) \qpsh\ on a maybe different set of 
definition. Finally, we use the supremum convolution to deduce 
a new characterisation for the \qpsc\ subsets in $\cbb^n$. 
\end{abstract}

\maketitle

\section{Introduction}

This article is a sequel and conclusion to our research previously 
developed in the works \cite{TPESZ} and \cite{TPESZ2}. Our main 
objective now is to characterise the \qpsh\ functions in terms of the 
theory of viscosity solutions; and this result is posed and proven in 
Definition~\ref{def-visc-qpsh}, Proposition~\ref{prop-usc-visc-qpsh}, 
and Theorem~\ref{thm-usc-visc-qpsh} in the following chapter. There 
we show that an \usc\ function $u$ is \qpsh\ if and only if either 
$u\equiv-\infty$ or its complex Hessian $H^{\cbb}u$ has at most $q$ 
strictly negative eigenvalues in the viscosity sense; i.e., if for 
every $\ccal^2$-smooth test function $\varphi$, the complex Hessian 
$H^{\cbb}\varphi$ has at most $q$ strictly negative eigenvalues at those 
points where $u-\varphi$ attains its maximum. This result was inspired 
by Alvarez, Lasry, and Lions paper \cite{ALL-convex}, where the authors 
showed that an \usc\ function $u$ is convex if and only if its real 
Hessian $H^{\rbb}u$ has no strictly negative eigenvalues in the 
viscosity sense. We strongly recommend the manuscripts of Crandall 
\cite{Crandall}, Crandall et al. \cite{CIL-viscosity}, Koike 
\cite{Koike}, or Katzourakis \cite{Katzourakis} for a general 
presentation on the theory of viscosity solutions.

On the other hand, in 1984 S{\l}odkowski introduced a product 
based supremum convolution for \qpsh\ functions $u$ in order to 
build non-increasing sequences of continuous \qpsh\ functions 
$\{u_k\}_{k=1}^\infty$ that converge pointwise to the original $u$, 
when $k\to\infty$, and such that each $u_k$ has almost everywhere 
second-order derivatives in the Peano sense; see \cite{Sl2}.

In the same form, one can apply the sup-convolution introduced by 
Moreau in 1963 \cite{Moreau1963,Moreau1966} in order to produce 
continuous approximations to \qpsh\ functions, because the \qpshy\ 
can be characterised in terms of viscosity subsolutions of the elliptic 
degenerate operator $\ominus$ given in Lemma~\ref{elliptic-degenerate} 
and one has the so called \textit{magic property}: the sup-convolution 
of a viscosity subsolution is again a viscosity subsolution. Besides, it 
is well known that the sup-convolution takes \usc\ functions $u$ (which 
may not be integrable) and in general produces continuous ones $u^\Phi$. 
For example, if $\Phi$ is of the form $-k\|{\cdot}\|^2$, the 
sup-convolution then satisfies $u^\Phi\geq{u}$, the sum 
$u^\Phi{+k}\|{\cdot}\|^2$ is convex, and $\{u^\Phi\}_{k=1}^\infty$ 
is a non-increasing sequence which converges pointwise to $u$ as 
$k\to\infty$; see Chapter~4 of \cite{Katzourakis} or Section~11 of 
\cite{Crandall}. We include a fast introduction to the sup-convolution 
and its properties in the third section of this work for the sake of 
completeness, but the reader who already knows the results can skip 
this section.

We conclude this work by showing in Proposition~\ref{prop-sup-conv} and 
the subsequent corollaries a variation of the \textit{magic property}: 
given any \qpsh\ function $u$ and a quadratic real-valued polynomial 
$g(y)=y^t\acal\overline{y}$ for a fixed Hermitian matrix $\acal$, the 
sup-convolution $[u{+g}]^\Phi{-g}$ is again \qpsh\ on a maybe different 
set of definition. As an application we deduce that the sup-convolution 
of a \sqpsh\ function is again \sqpsh; moreover, we also deduce a new 
characterisation for the \qpsc\ subsets in $\cbb^n$. This latter result 
extends a list of characterisations and properties of \qpsc\ sets we 
collected in \cite{TPESZ} and \cite{TPESZ2}.

\section{\qpsh\ functions as viscosity subsolutions}

We characterise in this section the \qpsh\ functions $u$ 
as viscosity subsolutions to the elliptic degenerate 
operator $u\mapsto\ominus(H^{\cbb}u)$ introduced in the 
Lemma~\ref{elliptic-degenerate} below; i.e., $u$ is \qpsh\ if and only 
if either $u\equiv-\infty$ or its complex Hessian $H^{\cbb}u$ has at 
most $q$ strictly negative eigenvalues in the viscosity sense. The main 
objective is to apply the results of the viscosity solutions theory 
originally developed by Crandall, Lions, and collaborators. We recommend 
in particular the manuscripts of Crandall \cite{Crandall}, Crandall et 
al. \cite{CIL-viscosity}, Koike \cite{Koike}, and Katzourakis 
\cite{Katzourakis}.

From now on $\langle{x,y}\rangle=\sum_{k=1}^n\Re(x_k\overline{y_k})$ 
denotes the standard real inner product between vectors $x$ and $y$ in 
$\cbb^n$, so that $\sqrt{\langle{x,x}\rangle}$ is the Euclidean norm 
$\|x\|$. In the same way, the notation $u\equiv-\infty$ means that $u$ 
is the function identically equal to $-\infty$ on its set of definition; 
the set $B_\rho(x)\subset\cbb^n$ is the standard Euclidean open ball of 
radius $\rho>0$ and with centre on $x$; while $U\subset\cbb^n$ always 
stands for a non-empty open set.

We need to recall the definition of a \qpsh\ function introduced by 
Hunt and Murray; see for example \cite{HM}. In particular, inspired 
by S{\l}odkowski's Lemma~\ref{sub-ph=(N-1)-psh} below, we prefer to 
introduce first the notion of \subph\ functions instead of the 
$(n{-1})$-\psh\ ones.

\begin{defn}\label{defqpsh} 
Let $u:U\to[-\infty,\infty)$ be an upper semi-continuous 
function defined on some open set $U\subset\cbb^n$.

\begin{enumerate}
\item\label{defqpsh-1} The function $u$ is \textit{\subph} on $U$ if and 
only if either $u\equiv-\infty$ on $U$ or $u+h$ has the maximum property 
for every continuous \psh\ function $h$ defined in a \nbh\ of any 
compact ball $\overline{B}\subset{U}$; i.e., the sum $u+h$ attains 
its maximum at the boundary~$bB$.

\item\label{defqpsh-2} For any integer $q$ with $0\leq{q}<n$, the 
function $u$ is \textit{\qpsh} on $U$ if and only if it is \subph\ on 
the intersection $U\cap\pi$ with every possible $(q{+}1)$-dimensional 
complex affine space $\pi$ in~$\cbb^n$. \Subph\ functions coincide with 
the $(n{-1})$-\psh\ ones.

\item\label{defqpsh-3} Finally, we say by convention that any \usc\ 
function is $N$-\psh\ whenever $N\geq{n}$.

\item The set of all \qpsh\ functions on $U$ is denoted by $\PSH_q(U)$. 

\end{enumerate}
\end{defn}

S{\l}odkowski showed in Lemma~4.4 of \cite{Sl} that conditions in the 
point~(\ref{defqpsh-1}) above can be relaxed; namely, it is sufficient 
to consider \ph\ polynomials (globally defined on $\cbb^n$) instead of 
\psh\ functions. Recall that a \ph\ polynomial (or function) is the real 
part $\pm\Re(h)$ of a holomorphic one $h$. We include S{\l}odkowski's 
lemma for the sake of completeness.

\begin{lem}\label{sub-ph=(N-1)-psh}{\rm\bf [S{\l}odkowski]}
Let $u:U\to[-\infty,\infty)$ be an \usc\ function well defined on 
a non-empty open set $U\subset\cbb^n$. We have that $u$ lies in 
$\PSH_{n-1}(U)$ if and only if the sum $u-\Re[\wp]$ has the local 
maximum property for every holomorphic polynomial $\wp:\cbb^n\to\cbb$; 
i.e., given any compact ball $\overline{B}$ in $U$, the new function 
$u-\Re[\wp]$ attains its maximum in the boundary $bB$. 
\end{lem}

\begin{proof}
See Lemma~4.4 and its proof in pages 122 and 123 of \cite{Sl}.
\end{proof}

The characterisation of \qpsh\ functions is simple when they are 
smooth enough, as it was indicated by Hunt, Murray and S{\l}odkowski. 
We present below a collection of properties of \qpsh\ functions, 
which can be found in various papers and are later used; see for 
example \cite{HM,Sl2,Fu2,Dieu,TPESZ,TPESZ2}.

\begin{prop}\label{propqpsh} 
Let $q$ and $r$ be a pair of non-negative integers, and $U$ be a 
non-empty open set in $\cbb^n$.

\begin{enumerate}
\item\label{qpsh-psh} 
The 0-\psh\ \fcts\ are the classical \psh\ ones. Moreover, 
$\PSH_q(U)$ is contained in $\PSH_r(U)$ whenever $q\leq{r}$.

\item\label{qpsh-local} 
The \qpshy\ is a local property; i.e. $u\in\PSH_q(U)$ if and only if 
every point $x\in{U}$ has a neighbourhood $\Omega\subset{U}$ such that 
the restriction $u|_\Omega$ lies in $\PSH_q(\Omega)$.

\item\label{qpsh-fcts} Given $c\geq0$ in $\mathbb{R}$ and two functions 
$u\in\PSH_q(U)$ and $v\in\PSH_r(U)$,
$$\begin{array}{ll}
cu\in\PSH_q(U),&\max\{u,v\}\in\PSH_{\max\{q,r\}}(U),\\[1pt]
u+v\in\PSH_{q+r}(U),&\min\{u,v\}\in\PSH_{q+r+1}(U).
\end{array}$$

\item\label{qpsh-decreasing} If $u_1\geq{u_2}\geq...$ is a 
non-increasing sequence of functions in $\PSH_q(U)$, then the point-wise 
limit $\displaystyle\lim_{k\to\infty}u_k$ lies in $\PSH_q(U)$.

\item\label{qpsh-supremum} Given a collection $\{u_j\}_{j\in{J}}$ in 
$\PSH_q(U)$ which is locally bounded from above at each point in~$U$, 
the upper semi-continuous regularisation 
$$z\mapsto\Big[\sup_{j\in{J}}u_j\Big]^\star\!(z)=
\limsup_{\zeta\to{z}}\Big[\sup_{j\in{J}}u_j(\zeta)\Big]
\quad\hbox{lies in}\quad\PSH_q(U).$$

\item\label{qpsh-loc-max} 
Assume that both $0\leq{q}<n$ and the open set $U\subset\cbb^n$ is 
bounded. Then, every function $u\in\PSH_q(U)$ that is \usc\ up to the 
boundary $bU$ also satisfies the maximum principle; i.e.,
$$\max_{\overline{U}}u\,=\,\max_{bU}u.$$

\item\label{qpsh-smooth} 
A $\mathcal{C}^2(U)$-smooth function $u$ is \qpsh\ on 
$U$ if and only if its complex Hessian below has at most $q$ strictly 
negative eigenvalues at each point $z\in{U}$,
\begin{equation}\label{complex-hessian}
H^{\cbb}u:=\left[\frac{\partial^2u}{\partial{z_k}
\partial\overline{z_\ell}}\right]_{k,\ell=1,\ldots,n}.
\end{equation}

\end{enumerate}
\end{prop}

We can now proceed with the main objective of this section; i.e, we 
characterise the \qpsh\ functions as those \usc\ functions whose complex 
Hessian has at most $q$ strictly negative eigenvalues in the viscosity 
sense. To do so, we need to show first that the number of strictly 
negative eigenvalues is an elliptic degenerate operator acting on the 
set of Hermitian matrices.

\begin{lem}\label{elliptic-degenerate}
Let $\ominus(\acal)$ be the number of strictly negative eigenvalues of 
the Hermitian matrix $\acal\in\cbb^{n\times{n}}$, and $\oplus(\acal)$ 
be its number of strictly positive eigenvalues. The operators $\ominus$ 
and $-\oplus$ are both elliptic degenerate; i.e., the inequalities 
below hold for all Hermitian matrices $\acal\geq\bcal$,
\begin{equation}\label{eqn-ellipt-degenerate}
\ominus(\acal)\leq\ominus(\bcal)\quad\hbox
{and}\quad\oplus(\acal)\geq\oplus(\bcal). 
\end{equation}
\end{lem}

Notice that neither the operator $\ominus$ nor $\oplus$ is continuous 
and that the latter sign ``$\geq$" designates the L\"owner partial 
order in the set of Hermitian matrices; i.e., $\acal\geq\bcal$ whenever 
$\acal{-\bcal}$ is positive semidefinite.

\begin{proof} This is a simple consequence of Theorem~7.9 of 
\cite{Zhang}, because for any two Hermitian matrices $\acal\geq\bcal$ 
with ordered eigenvalues: 
$$\lambda_1(\acal)\geq\ldots\geq\lambda_n(\acal)\quad\hbox
{and}\quad\lambda_1(\bcal)\geq\ldots\geq\lambda_n(\bcal),$$ 
the inequalities $\lambda_k(\acal)\geq\lambda_k(\bcal)$ 
hold for all indices $k=1,\ldots,n$. 
\end{proof}

As in the classical theory of viscosity solutions, the definition for 
\qpsh\ functions in the viscosity sense is naturally motivated by the 
following result for $\ccal^2$-differentiable functions.

\begin{lem}\label{lem-smooth-visc-qpsh} 
For any integer $q\geq0$, a function $u\in\mathcal{C}^2(U)$ is \qpsh\ on 
a non-empty open set $U\subset\cbb^n$ if and only if for every function 
$\varphi\in\ccal^2(U)$, the operator $\ominus(H^{\cbb}\varphi)\leq{q}$ 
at those points $\hat{p}\in{U}$ where $u-\varphi$ attains its maximum. 
\end{lem}

\begin{proof} Firstly assume that for every $\varphi\in\ccal^2(U)$, 
the operator $\ominus(H^{\cbb}\varphi)\leq{q}$ at those points 
$\hat{p}\in{U}$ where $u-\varphi$ attains its maximum. Since 
$u\in\mathcal{C}^2(U)$, we can fix $\varphi\equiv{u}$, so that 
$u-\varphi$ obviously attains its maximum (zero) at every point 
$z\in{U}$. Whence, $u\equiv\varphi$ lies in $\PSH_q(U)$, because 
$\ominus(H^{\cbb}\varphi)\leq{q}$ at every $z\in{U}$ and according 
to (\ref{qpsh-smooth}) in Proposition~\ref{propqpsh}.

On the other hand, suppose now that $u\in\ccal^2(U)$ is \qpsh\ and take 
$\varphi\in\ccal^2(U)$, such that the difference $g=u-\varphi$ attains 
its maximum at some point $\hat{p}\in{U}$. Hence, when they are 
evaluated at $\hat{p}$, both the gradient $\nabla{g}=0$ vanishes and the 
real Hessian $H^{\rbb}g\leq0$ is negative semi-definite. We assert that 
the complex Hessian $H^{\cbb}\varphi$ in (\ref{complex-hessian}) has at 
most $q$ strictly negative eigenvalues at $\hat{p}$.

Recall that the real Hessian $H^{\rbb}g\in\rbb^{2n\times{2n}}$ is 
symmetric and Hermitian, so that the condition $H^{\rbb}g\leq0$ implies 
$\overline{v}^t[H^{\rbb}g]v\leq0$ for every vector $v\in\cbb^{2n}$. We 
can understand the relation between the real and complex Hessians by 
expanding the complex coordinates into real ones $z_k=x_k{+i}y_k$, so 
that $\cbb^n\cong\rbb^{2n}$. The following identities are easily 
satisfied for $z=\hat{p}$ and all vectors $\xi\in\cbb^n$,
\begin{eqnarray*}
\overline{\xi^t}\left[\frac{\partial^2g}
{\partial{z_k}\partial\overline{z_\ell}}\right]_{k,\ell}\!\xi
&=&\frac14\overline{\big[\xi^t,i\xi^t\big]}\begin{bmatrix}
\Big[\frac{\partial^2g}{\partial{x_k}\partial{x_\ell}}\Big]_{k,\ell}&
\Big[\frac{\partial^2g}{\partial{x_k}\partial{y_\ell}}\Big]_{k,\ell}\\[1ex] 
\Big[\frac{\partial^2g}{\partial{y_k}\partial{x_\ell}}\Big]_{k,\ell}&
\Big[\frac{\partial^2g}{\partial{y_k}\partial{y_\ell}}\Big]_{k,\ell}
\end{bmatrix}\begin{bmatrix}\xi\\{i}\xi\end{bmatrix}\\
&=&\frac14\overline{\big[\xi^t,i\xi^t\big]}\,\Big[H^{\rbb}g\Big]
\begin{bmatrix}\xi\\{i}\xi\end{bmatrix}\;\leq\;0.
\end{eqnarray*}

Thus, both inequalities $H^{\cbb}g\leq0$ and 
$H^{\cbb}u\leq{H^\cbb}\varphi$ hold, when the complex Hessians are 
all evaluated at $\hat{p}$. Notice that $\ominus(H^{\cbb}u)\leq{q}$ 
on $U$, because $u$ is \qpsh\ in $U$ and point (\ref{qpsh-smooth}) 
in Proposition~\ref{propqpsh}. A direct application of 
Lemma~\ref{elliptic-degenerate} yields the desired result: 
$\ominus(H^{\cbb}\varphi)\leq{q}$ at those points $\hat{p}$ 
where $u-\varphi$ attains its maximum. \end{proof}

We may now present the main definition of this work, the main idea is to 
rephrase the previous result, so as to get a proper definition for \usc\ 
functions instead of $\ccal^2$-differentiable ones.

\begin{defn}\label{def-visc-qpsh} 
Let $q\geq0$ be an integer, and $u:U\to[-\infty,\infty)$ be an upper 
semi-continuous function defined on some open set $U\subset\cbb^n$. We 
say that $u$ is \textit{\qpsh\ in the viscosity sense on $U$} whenever 
either $u\equiv-\infty$ on $U$ or for every $\varphi\in\ccal^2(U)$, the 
operator $\ominus(H^{\cbb}\varphi)\leq{q}$ at those points $\hat{p}$ 
where $u-\varphi$ attains its maximum; i.e., the complex Hessian 
$H^{\cbb}\varphi$ in (\ref{complex-hessian}) has at most $q$ strictly 
negative eigenvalues at those points $\hat{p}$.

In particular, if $u\not\equiv-\infty$, we must only consider those test 
functions $\varphi$ for which $u-\varphi$ can indeed attain its maximum 
at some $\hat{p}\in{U}$, so that $(u{-}\varphi)(\hat{p})$ and 
$u(\hat{p})$ are both (finite) real numbers.

For example, any upper semi-continuous function $u$ is \qpsh\ 
in the viscosity sense if $u\equiv-\infty$ or $q\geq{n}$, 
since $H^{\cbb}\varphi$ in (\ref{complex-hessian}) is an 
$[n{\times}n]$-Hermitian matrix when $\varphi\in\ccal^2(U)$. 
\end{defn}

We now devote the rest of this section to prove that the 
\textit{\qpshies} in the viscosity and classical senses coincide 
for general upper semi-continuous functions. We begin presenting 
the simpler implication.

\begin{prop}\label{prop-usc-visc-qpsh}   
Every function $u\in\PSH_q(U)$ is \qpsh\ in the viscosity sense on 
$U$, for any integer $q\geq0$ and non-empty open set $U\subset\cbb^n$. 
\end{prop}

\begin{proof} The result is trivial when $q\geq{n}$ or $u\equiv-\infty$ 
on $U$, because $u$ is \qpsh\ in the viscosity sense in both cases. 
Thus, we assume from now on that $q<n$ and $u(\hat{x})\in\rbb$ is 
finite for some $\hat{x}\in{U}$. We prove the result by assuming 
that $u$ is not \qpsh\ in the viscosity sense on $U$. Whence, 
following Definition~\ref{def-visc-qpsh} there is a test function 
$\varphi\in\ccal^2(U)$ such that $u-\varphi$ attains its maximum at 
some $\hat{p}\in{U}$, but the operator $\ominus(H^{\cbb}\varphi)>q$ at 
the point $\hat{p}$. In particular, $u(\hat{p})\in\rbb$ is finite and 
the negative complex Hessian $-H^{\cbb}\varphi$ has at least $q{+1}$ 
strictly positive eigenvalues at $\hat{p}$, so that $-H^{\cbb}\varphi$ 
has the same number of positive eigenvalues at every point in some 
small neighbourhood $\Omega_1$ of $\hat{p}$. According to point 
(\ref{qpsh-smooth}) in Proposition~\ref{propqpsh} this means that 
$-\varphi$ is $(n{-q-}1)$-\psh\ on $\Omega_1$.

Choose $\beta>0$ small enough such that 
$-\varphi(z){-}\beta\|z{-}\hat{p}\|^2$ is $(n{-q-}1)$-\psh\ on a maybe 
smaller neighbourhood $\Omega_2$ of $\hat{p}$; i.e., its complex Hessian 
has at least $q{+1}$ positive eigenvalues on $\Omega_2$. The sum 
$(u{-}\varphi)(z){-}\beta\|z{-}\hat{p}\|^2$ then attains its strict 
global maximum at $\hat{p}\in\Omega_2$ and it is $(n{-1})$-\psh\ on 
$\Omega_2$. This result is a consequence of point~(\ref{qpsh-fcts}) 
in Proposition \ref{propqpsh}, after considering that $u$ lies in 
$\PSH_q(U)$, that the evaluation $u(\hat{p})\in\rbb$ is finite, and that 
$u-\varphi$ attains its maximum at $\hat{p}$. The existence of a strict 
global maximum inside $\Omega_2$ contradicts the local maximum principle 
stated in the point (\ref{qpsh-loc-max}) of Proposition~\ref{propqpsh}, 
so that $u$ must be \qpsh\ in the viscosity sense on $U$. 
\end{proof}

We need the following result to prove the opposite implication to 
the previous proposition. This is a classical result in the theory 
of viscosity solutions, but we include its proof for the sake of 
completeness.

\begin{lem}\label{lem-extension}
Let $\Omega\subset{U}\subset\cbb^n$ be non-empty open sets, and 
$u:U\to[-\infty,\infty)$ be an upper semi-continuous function. 
Assume there is a function $\varphi\in\ccal^2(\Omega)$ such that the 
restriction of $u-\varphi$ to $\Omega$ attains its maximum at some point 
$\hat{p}\in\Omega$ and $u(\hat{p})\in\rbb$ is finite. Then there is a 
second function $\psi\in\ccal^2(U)$ such that $u-\psi$ has a strict 
global maximum at $\hat{p}$ in $U$ and its evaluation $\psi(\hat{p})$, 
gradient $\nabla\psi(\hat{p})$, and complex Hessian 
$H^{\cbb}\psi(\hat{p})$ coincide with the respective evaluation 
$\varphi(\hat{p})$, gradient $\nabla\varphi(\hat{p})$, and complex 
Hessian $H^{\cbb}\varphi(\hat{p})$ of the original function. 
\end{lem}

\begin{proof} Define the function 
$\psi_0(z):=\varphi(z)+\|z{-}\hat{p}\|^6$. One can easily see that 
the restriction of $u-\psi_0$ to $\Omega$ has a strict global maximum 
at $\hat{p}\in\Omega$ and its evaluation $\psi_0(\hat{p})$, gradient 
$\nabla\psi_0(\hat{p})$, and complex Hessian $H^{\cbb}\psi_0(\hat{p})$ 
coincide with the respective evaluation $\varphi(\hat{p})$, gradient 
$\nabla\varphi(\hat{p})$, and complex Hessian $H^{\cbb}\varphi(\hat{p})$ 
of the original function. We only need to extend $\psi_0$ as a 
$\ccal^2$-smooth function onto $U$ to conclude the proof. We do 
this by steeps.

Also define $V_k:=\{x\in{U}:u(x)<k\}$, so that 
$U=\bigcup_{k=1}^{\infty}V_k$. Since $u$ is upper semi-continuous, we 
have that $V_k$ is an open subset of $\cbb^n$ for each integer $k\geq1$. 
Take any compact ball $\overline{B}\subset\Omega$ with centre at 
$\hat{p}$, so that $\psi_0$ is $\ccal^2$-smooth in a neighbourhood 
of $\overline{B}$ and the restriction of $u-\psi_0$ to $\overline{B}$ 
attains its strict global maximum at $\hat{p}\in{B}$. Since $u$ is 
upper semi-continuous on $U$, it is bounded from above on the compact 
ball $\overline{B}$ by a large enough integer $k\geq1$, so that 
$\overline{B}\subset{V_k}$. All the previous properties imply 
that $\psi_0$ can be extended from $B$ onto a neighbourhood of 
$\overline{V_k}$ as a $\ccal^2$-smooth function $\psi_k$ such 
that the restriction of $u-\psi_k$ to $\overline{V_k}$ also 
has its strict global maximum at $\hat{p}$.

We can obviously carry on with this process, extending $\psi_j$ from 
$V_j$ onto a neighbourhood of $\overline{V_{j+1}}$ as a $\ccal^2$-smooth 
function $\psi_{j+1}$ such that the restriction of $u-\psi_{j+1}$ to 
$\overline{V_{j+1}}$ attains its strict global maximum at $\hat{p}$. 
Since $U=\bigcup_{k=0}^{\infty}V_k$, the above extension process yields 
a function $\psi\in\ccal^2(U)$, such that $u-\psi$ also has its strict 
global maximum at $\hat{p}$ and the functions $\psi$ and $\psi_0$ 
coincide inside the open ball $B$ centred at $\hat{p}$, as we 
wanted to prove. 
\end{proof}

A direct consequence of this lemma is that the \qpshy\ in the viscosity 
sense is a local property.

\begin{cor}\label{cor-visc-qpsh-local}
Let $q\geq0$ be an integer, and $u:U\to[-\infty,\infty)$ be an 
upper semi-continuous function defined on a non-empty open set 
$U\subset\cbb^n$. \begin{enumerate}

\item\label{visc-qpsh-local-1} 
If every point $x\in U$ has an open neighbourhood $\Omega\subset{U}$ 
such that the restriction $u|_\Omega$ is \qpsh\ in the viscosity sense 
on $\Omega$, then $u$ is also \qpsh\ in the viscosity sense on $U$.

\item\label{visc-qpsh-local-2} 
Given any non-empty open subset $\Omega\subset{U}$, if $u$ is \qpsh\ in 
the viscosity sense on $U$, the restriction $u|_\Omega$ is then \qpsh\ 
in the viscosity sense on $\Omega$. \end{enumerate} 
\end{cor}

\begin{proof} Statement in point~(\ref{visc-qpsh-local-1}) trivially 
holds when $u\equiv-\infty$ on $U$; otherwise, this point is proved by 
taking any function $\psi\in\ccal^2(U)$ such that $u-\psi$ attains its 
maximum at some point $\hat{p}\in{U}$ with $u(\hat{p})\in\rbb$. Let 
$\Omega$ be an open neighbourhood of $\hat{p}$ in $U$ such that the 
restriction $u|_\Omega$ is \qpsh\ in the viscosity sense in $\Omega$. 
Since the restriction of $u-\psi$ to $\Omega$ also attains its maximum 
at $\hat{p}\in\Omega$, Definition~\ref{def-visc-qpsh} yields that the 
operator $\ominus(H^{\cbb}\psi)\leq{q}$ at $\hat{p}$, and so $u$ is 
\qpsh\ in the viscosity sense in $U$ by the same definition. 


Result in point~(\ref{visc-qpsh-local-2}) trivially holds when 
$u\equiv-\infty$ on $\Omega$; otherwise, this point is proved by taking 
any function $\varphi\in\ccal^2(\Omega)$ such that $u|_\Omega-\varphi$ 
attains its maximum at some point $\hat{p}\in\Omega$ with 
$u(\hat{p})\in\rbb$. Lemma~\ref{lem-extension} yields the existence of 
$\psi\in\ccal^2(U)$ such that $u-\psi$ has a strict global maximum at 
$\hat{p}\in{U}$ and the complex Hessians $H^{\cbb}\psi(\hat{p})$ and 
$H^{\cbb}\varphi(\hat{p})$ coincide. Since $u$ is \qpsh\ in the 
viscosity sense on $U$, the Hessians $H^{\cbb}\varphi=H^{\cbb}\psi$ 
have at most $q$ strictly negative eigenvalues at $\hat{p}$, and so 
$u|_\Omega$ is also \qpsh\ in the viscosity sense on $\Omega$. 
\end{proof}

We may now prove that the \textit{\qpshies} in the viscosity and 
classical senses coincide for general upper semi-continuous functions.
This result was inspired by Alvarez, Lasry, and Lions paper 
\cite{ALL-convex}, where they shown that an \usc\ function $u$ is 
convex if and only if its real Hessian $H^{\rbb}u$ has no strictly 
negative eigenvalues in the viscosity sense.

\begin{thm}\label{thm-usc-visc-qpsh}
Let $q\geq0$ be any integer, and $U\subset\cbb^n$ be a non-empty open 
set. An upper semi-continuous function $u:U\to[-\infty,\infty)$ is 
\qpsh\ on $U$ if and only if it is \qpsh\ in the viscosity sense on $U$. 
\end{thm}

\begin{proof} The result is trivial when $q\geq{n}$ or 
$u\equiv-\infty$ on $U$. Moreover, it is already shown in 
Proposition~\ref{prop-usc-visc-qpsh} that every $u\in\PSH_q(U)$ 
is \qpsh\ in the viscosity sense. We assume from now on that both 
$u\not\equiv-\infty$ is \qpsh\ in the viscosity sense on $U$ and $q<n$. 
We show that $u\in\PSH_q(U)$ by considering the cases introduced in the 
main Definition~\ref{defqpsh}. The first case happens when $q$ is equal 
to $n{-1}$; and so we introduce the test function
\begin{equation}\label{eqn1-usc-visc-qpsh}
\varphi_\beta(z):=\Re[\wp](z)-\beta\|z\|^2
\quad\hbox{for}\quad{z}\in\cbb^n,  
\end{equation}
where $\wp:\cbb^n\to\cbb$ is an arbitrary holomorphic 
polynomial and $\beta\geq0$ is a real number. According to 
Lemma~\ref{sub-ph=(N-1)-psh}, $u$ lies in $\PSH_{n-1}(U)$, whenever 
the maxima $\theta_1\geq\theta_2$ below are equal for $\beta=0$ and 
any non-trivial compact ball $\overline{B}\subset{U}$ and arbitrary 
holomorphic polynomial $\wp$,
\begin{equation}\label{eqn2-usc-visc-qpsh}
\theta_1:=\max_{\overline{B}}(u-\varphi_0)\quad
\hbox{and}\quad\theta_2:=\max_{bB}(u-\varphi_0).
\end{equation}

Notice that both $\theta_1$ and $\theta_2$ exist, because $u$ 
is \usc, but $\theta_1$ or $\theta_2$ may be equal to $-\infty$. 
Suppose by way of contradiction that $\theta_1>\theta_2$, so that the 
restriction of $u-\varphi_0$ to $\overline{B}$ does not attain its 
maximum $\theta_1\in\rbb$ at the boundary $bB$. One can easily choose 
$\beta>0$ small enough, so that the restriction of $u-\varphi_\beta$ 
to $\overline{B}$ does not attain its maximum at $bB$ either, but this 
maximum $\theta_3\in\rbb$ is indeed attained at some point $\hat{p}$ 
in the interior $B$. Lemma~\ref{lem-extension} yields the existence 
of $\psi\in\ccal^2(U)$ such that $u-\psi$ also has its strict global 
maximum at $\hat{p}\in{U}$ and the complex Hessians 
$H^{\cbb}\psi(\hat{p})$ and $H^{\cbb}\varphi_\beta(\hat{p})$ coincide. 
Since $u$ is $(n{-1})$-\psh\ in the viscosity sense on $U$, the Hessians 
$H^{\cbb}\psi=H^{\cbb}\varphi_\beta$ must have at most $n{-1}$ strictly 
negative eigenvalues at the point $\hat{p}$ according to 
Definition~\ref{def-visc-qpsh}. We obviously have a contradiction, 
because $H^{\cbb}\varphi_\beta$ in (\ref{eqn1-usc-visc-qpsh}) is 
equal to $-\beta<0$ times the $[n{\times}n]$-identity matrix. Whence, 
$\theta_1=\theta_2$ in (\ref{eqn2-usc-visc-qpsh}), and so $u$ lies 
in $\PSH_{n-1}(U)$ according to Lemma~\ref{sub-ph=(N-1)-psh}.

We now consider the case when $0\leq{q}\leq{n-}2$. Let 
$\pi\subset\cbb^n$ be any $(q{+1})$-di\-men\-sio\-nal complex affine 
space such that $U\cap\pi$ is not empty. We can assume without loss 
of generality that $\pi$ is of the form $\cbb^{q+1}{\times}\{0\}$ 
for some particular system of coordinates $(x,y)$ in 
$\cbb^{q+1}{\times}\cbb^{n-q-1}$, since the eigenvalues of complex 
Hessians are invariant under holomorphic rotations and translations. 
Obviously, we have nothing to prove when $u$ is identically equal to 
$-\infty$ in $U\cap\pi$; otherwise, we follow the ideas presented in 
the previous paragraphs. Thus, $u\in\PSH_q(U)$, whenever the maxima 
$\theta_1\geq\theta_2$ below are equal for any holomorphic polynomial 
$\wp:\cbb^n\to\cbb$ and non-empty open ball $B\subset\cbb^{q+1}$ with 
product $\overline{B}{\times}\{0\}$ contained in $U\cap\pi$,
\begin{equation}\label{eqn3-usc-visc-qpsh}
\theta_1:=\max_{\overline{B}\times\{0\}}\Upsilon\quad
\hbox{and}\quad\theta_2:=\max_{bB\times\{0\}}\Upsilon
\quad\hbox{for}\quad\Upsilon:=u-\Re[\wp].
\end{equation}

Notice that both $\theta_1$ and $\theta_2$ exist, because $u$ 
is \usc, but $\theta_1$ or $\theta_2$ may be equal to $-\infty$. 
Suppose by way of contradiction that $\theta_1>\theta_2$, so that the 
restriction of $\Upsilon$ to $\overline{B}{\times}\{0\}$ does not attain 
its maximum $\theta_1\in\rbb$ at the subset $bB{\times}\{0\}$. Since 
$\Upsilon$ is also \usc, we can find a second non-empty open ball $D$ 
in $\cbb^{n-q-1}$ with centre at the origin and radius small enough, 
such that
$$\overline{B{\times}D}\subset{U}\quad\hbox{and}
\quad\max_{\overline{B{\times}D}}\Upsilon>\theta_1>
\max_{bB\times\overline{D}}\Upsilon>\theta_2,$$
where $\Upsilon=u-\Re[\wp]$. Now define for 
$z=(x,y)\in\cbb^{q+1}{\times}\cbb^{n-q-1}$,
\begin{equation}\label{eqn4-usc-visc-qpsh} 
\varphi(x,y):=\Re[\wp](x,y)-\beta\|x\|^2+\delta\|y\|^2. 
\end{equation}

Again, since $u$ is \usc, we can choose $\beta>0$ small enough and 
$\delta>0$ large enough, so that the restriction of $u-\varphi$ to 
$\overline{B{\times}D}$ does not attain its maximum at the boundary 
$b(B{\times}D)$ either, but this maximum $\theta_3\in\rbb$ is indeed 
attained at some point $\hat{p}$ in the interior $B{\times}D$. 
Lemma~\ref{lem-extension} yields the existence of $\psi\in\ccal^2(U)$ 
such that $u-\psi$ also has its strict global maximum at $\hat{p}\in{U}$ 
and the complex Hessians $H^{\cbb}\psi(\hat{p})$ and 
$H^{\cbb}\varphi(\hat{p})$ coincide. Since $u$ is \qpsh\ in the 
viscosity sense on $U$, the Hessians $H^{\cbb}\psi=H^{\cbb}\varphi$ 
must have at most $q$ strictly negative eigenvalues at the point 
$\hat{p}$ according to Definition~\ref{def-visc-qpsh}. We 
obviously have a contradiction, because $H^{\cbb}\varphi(\hat{p})$ in 
(\ref{eqn4-usc-visc-qpsh}) has $q{+1}$ strictly negative eigenvalues 
coming from the term $-\beta\|x\|^2$. Hence, $\theta_1=\theta_2$ in 
(\ref{eqn3-usc-visc-qpsh}), and so $u$ lies in $\PSH_q(U)$ according 
to Definition~\ref{defqpsh}.
\end{proof}

\section{Basic properties of the sup-convolution} 

We have so far characterised the \qpsh\ functions as \usc\ viscosity 
subsolutions to the elliptic degenerate operator $\ominus$ introduced 
in Lemma~\ref{elliptic-degenerate}. One of the main consequences of 
this result is that we can apply the theory of viscous solutions into 
the field of several complex variables, as it was initially done by 
S{\l}odkowski and Tomassini in \cite{SlTo}, but we must be careful, 
because the operator $\ominus$ is not continuous. We recommend again 
to the interested reader the manuscripts of Crandall \cite{Crandall}, 
Crandall et al. \cite{CIL-viscosity}, Koike \cite{Koike}, and 
Katzourakis \cite{Katzourakis}, for a general introduction to 
the theory of viscous solutions.

One of the most important tools in the theory of viscosity solutions 
is the sup-convolution introduced by Moreau in 1963 \cite{Moreau1963, 
Moreau1966}. It plays the role represented by the classical integral 
convolution in the theory of linear differential equations. Actually, 
it is well known that the sup-convolution takes an \usc\ function $u$ 
(which may not be integrable) and in general produces a continuous one 
$u^\Phi$. In particular, if $\Phi$ is of the form $-k\|{\cdot}\|^2$, 
the sup-convolution then satisfies $u^\Phi\geq{u}$, the sum 
$u^\Phi{+k}\|{\cdot}\|^2$ is convex, and the sequence 
$\{u^\Phi\}_{k=1}^\infty$ is non-increasing and converges pointwise to 
$u$ as $k\to\infty$; see for example Chapter~4 of \cite{Katzourakis} 
or Section~11 of \cite{Crandall}. Moreover, one also has the so called 
\textit{magic property}: the sup-convolution of a viscosity subsolution 
is again a viscosity subsolution.

We include in this chapter a fast introduction to the sup-convolution 
and its properties for the sake of completeness, but the reader who 
already knows the results can skip this chapter. Recall the original 
Moreau definition \cite{Moreau1963,Moreau1966}.

\begin{defn}\label{def-sup-conv}
Let $u:X\to[-\infty,\infty)$ and $\Phi:\cbb^n\to[-\infty,\infty)$ be 
a pair of \usc\ functions, where $X\subset\cbb^n$ is a non-empty set. 
The \textit{sup-convolution} is defined below for every $y\in\cbb^n$,
\begin{equation}\label{eqn-sup-conv}
{u_X^\Phi}(y):=\sup_{x\in{X}}
\big\{u(x)+\Phi(y{-x})\}\in[-\infty,\infty].
\end{equation}

We simply write $u^\Phi$ instead of $u_X^\Phi$ when it is clear 
on which set the function $u$ is defined. 
\end{defn}

Even when the sup-convolution is well defined for every $y\in\cbb^n$, 
the most interesting results appear at those points $y\in\cbb^n$ for 
which the supremum in (\ref{eqn-sup-conv}) can be attained at some 
$\hat{x}\in{X}$. Thus, we recall the semiconvex functions and the 
proper sets of definition for the sup-convolution.

\begin{defn}\label{def-proper-set} 
Let $u:X\to[-\infty,\infty)$ and $\Phi:\cbb^n\to[-\infty,\infty)$ be 
both \usc, where $X\subset\cbb^n$ is non-empty. The \textit{proper set 
of definition} for the sup-convolution $u_X^\Phi$ is composed by those 
points $y\in\cbb^n$ for which there is $\hat{x}\in{X}$, such that the 
supremum $u_X^\Phi(y)$ in (\ref{eqn-sup-conv}) is equal to 
$\Phi({y-}\hat{x}){+}u(\hat{x})$.

Besides, a function $\Psi:\cbb^n\to\rbb\cup\{\infty\}$ is said to be 
\textit{semiconvex with constant $\delta>0$} if and only if the sum 
$\Psi{+}\delta\|{\cdot}\|^2$ is convex on $\cbb^n$.
\end{defn}

Notice that the classical definition for convexity is extended 
to consider functions with image in $\rbb\cup\{\infty\}$. The 
sup-convolution has some simple but interesting properties. 
The proofs are included for the sake of completeness.

\begin{lem}\label{lem-sup-conv} 
Let $u:X\to[-\infty,\infty)$ and $\Phi:\cbb^n\to[-\infty,\infty)$ be a 
pair of \usc\ functions, where $X\subset\cbb^n$ is a non-empty set. The 
statements below hold up for the sup-convolution $u^\Phi=u_X^\Phi$ in 
(\ref{eqn-sup-conv}).
\begin{enumerate}
\item\label{lem-sup-conv-1} The inequality $u\leq{u}^\Phi$ 
is satisfied on $X$, whenever $\Phi(0)=0$.

\item\label{lem-sup-conv-2} If $\Phi\leq0$ on $\cbb^n$ and there is 
$M\in\rbb$ with $u\leq{M}$ on $X$, then $u^\Phi\leq{M}$ is also bounded 
from above on $\cbb^n$ by the same constant.

\item\label{lem-sup-conv-3} One has that $u_X^\Phi>-\infty$ on $\cbb^n$ 
if and only if neither $\Phi$ nor $u$ is identically equal to $-\infty$ 
on their respective set of definition.

\item\label{lem-sup-conv-4} Let $Y\subset\cbb^n$ be the proper set 
of definition for $u_X^\Phi$. If neither $\Phi$ nor $u$ is identically 
equal to $-\infty$, then the image $u_X^\Phi(Y)\subset\rbb$ does not 
contain the extremal values $\pm\infty$.

\item\label{lem-sup-conv-5} If $\Phi\in\ccal(\cbb^n)$ is semiconvex 
with constant $\delta>0$ on $\cbb^n$ and $u\not\equiv-\infty$, then 
$u_X^\Phi$ is also semiconvex with constant $\delta>0$ on $\cbb^n$.

\item\label{lem-sup-conv-7} Let $f:\rbb\to[-\infty,\infty)$ be \usc\ 
and monotonically decreasing. Assume that $X\subset\cbb^n$ is closed,
$\mathrm{dist}_X$ is the Euclidean distance to $X$, and 
$\Phi=f(\|{\cdot}\|)$. The following identities hold on $y\in\cbb^n$,
\begin{equation}\label{charac-funct-1}
\chi_{\cbb^n}^\Phi=\chi_X^\Phi=f\circ
\mathrm{dist}_X\quad\hbox{for}\quad\chi(y):=
\left\{\begin{array}{cl}0&\hbox{if~}y\in{X};\\
-\infty&\hbox{otherwise}.\end{array}\right.
\end{equation}
\end{enumerate}
\end{lem}

Recall that $\langle{x,y}\rangle=\sum_{k=1}^n\Re(x_k\overline{y_k})$ 
denotes the standard real inner product between vectors $x$ and $y$ 
in $\cbb^n$. Moreover, the notation $u\equiv-\infty$ means that $u$ 
is identically equal to $-\infty$ on its set of definition $X$.

\begin{proof} The inequality $u\leq{u^\Phi}$ in statement 
(\ref{lem-sup-conv-1}) holds up after fixing $x=y$ into formula 
(\ref{eqn-sup-conv}). The inequality $u^\Phi\leq{M}$ in point 
(\ref{lem-sup-conv-2}) is proved by setting $u\leq{M}$ into 
(\ref{eqn-sup-conv}), so that for every $y\in\cbb^n$ 
$$u^\Phi(y)\leq\sup_{x\in{X}}\big\{M+\Phi(y{-x})\big\}\leq{M}.$$

The result in point (\ref{lem-sup-conv-3}) follows from the fact 
that the sup-convolution $u^\Phi$ is calculated with the supremum. 
The result in statement (\ref{lem-sup-conv-4}) is proved by cases. 
Given any point $y$ in the proper set of definition for $u_X^\Phi$, 
there is $\hat{x}\in{X}$ such that $u_X^\Phi(y)$ is equal to 
$u(\hat{x}){+}\Phi({y-}\hat{x})$. Thus, $u_X^\Phi(y)<\infty$, because 
$u<\infty$ and $\Phi<\infty$ on their respective sets of definition. 
Besides, point~(\ref{lem-sup-conv-3}) yields that $u_X^\Phi>-\infty$ 
on $\cbb^n$, because neither $\Phi$ nor $u$ is identically equal to 
$-\infty$.

Definition~\ref{def-proper-set} and the semiconvexity in the 
hypotheses of statement (\ref{lem-sup-conv-5}) easily yield that 
$\Phi\not\equiv-\infty$ and the sum $\Phi{+}\delta\|{\cdot}\|^2$ 
is convex on $\cbb^n$ for some $\delta>0$. Hence, we also have that 
$u_X^\Phi(y){+}\delta\|y\|^2$ is convex with respect to the variable 
$y\in\cbb^n$, because it is the supremum of the following convex 
functions,
\begin{eqnarray*}
y&\mapsto&u(x)-\delta\|x\|^2+2\delta\langle 
{x,y}\rangle+\Phi(y{-x})+\delta\|y{-x}\|^2\\ 
&&=\;u(x)+\Phi(y{-}x)+\delta\|y\|^2, 
\end{eqnarray*} 
which are all indexed by those points $x\in{X}$ with $u(x)\in\rbb$. 
Finally, point~(\ref{lem-sup-conv-7}) easily follows from the definition 
of the distance function and the hypothesis that $f$ is monotonically 
decreasing; i.e.,
\begin{equation}\label{eqn2-lem-sup-conv} 
\chi_{\cbb^n}^\Phi(y)=\sup_{x\in\cbb^n}\big\{\chi(x)+f(\|y{-x}\|) 
\big\}=\sup_{x\in{X}}f(\|y{-x}\|)=f(\mathrm{dist}_X(y)). 
\end{equation} 
\end{proof}

As we have said, one of the properties of the sup-convolution is that it 
takes upper semi-continuous functions (which may not be integrable) and 
produces continuous functions. This property is obtained when $\Phi$ is 
semiconvex; see Definition~\ref{def-proper-set}.

\begin{lem}\label{lem-convex-funct}
Let $X\subset\cbb^n$ be non-empty, and $u:X\to[-\infty,\infty)$ be upper 
semi-continuous. Given a semiconvex function $\Phi\in\ccal(\cbb^n)$ with 
constant $\delta>0$, so that $\Phi\not\equiv-\infty$, assume the proper 
set of definition for the sup-convolution $u_X^\Phi$ has non-empty 
interior $W\neq\emptyset$. The restriction of 
$f=u_X^\Phi{+\delta}\|{\cdot}\|^2$ to $W$ is then locally convex and 
continuous. Besides, if $u\not\equiv-\infty$, the function $f$ is twice 
differentiable almost everywhere on $W$ in the sense of Alexandrov; 
i.e., for almost every $x\in{W}$ there are a vector $\nu\in\cbb^n$ 
and a symmetric matrix $\acal\in\rbb^{2n\times2n}$ such that
$$\lim_{h\to0}\frac{f(x{+}h)-f(x)-\langle\nu, 
h\rangle-(h^t\!\acal{h})/2}{\|h\|^2}=0.$$ 
\end{lem}

In particular, if $f$ is $\mathcal{C}^2$-differentiable on a 
small neighbourhood of $x$ in $W$, one automatically has that 
$\nu=\nabla{f}(x)$ is the gradient and $\acal=H^{\rbb}f(x)$ is the real 
Hessian; moreover $\acal\geq0$, because $f$ is locally convex at $x$.

\begin{proof} The result is trivial when $u\equiv-\infty$, because 
$u_X^\Phi\equiv-\infty$ would be constant. Thus, we suppose from now 
on that it is not the case. Definition~\ref{def-sup-conv} and the 
point~(\ref{lem-sup-conv-5}) in Lemma~\ref{lem-sup-conv} imply that 
$f=u_X^\Phi{+\delta}\|{\cdot}\|^2$ is convex on $\cbb^n$, but it may 
take the extremal values $\pm\infty$. The point~(\ref{lem-sup-conv-4}) 
in the same Lemma~\ref{lem-sup-conv} then yields that $f(W)\subset\rbb$ 
does not contain the extremal values $\pm\infty$, so that the 
restriction of $f$ to $W$ is locally convex and continuous. The 
latter results are a consequence of Alexandrov's theorem for 
convex functions; see \cite{Alexandrov,CIL-viscosity}. 
\end{proof}

Even when the sup-convolution $u^\Phi$ can be defined for all \usc\ 
functions $u$ and $\Phi$, the more interesting cases happen when the 
proper set of definition for $u^\Phi$ is not empty and $u^\Phi$ is 
itself a continuous function. These properties are partially satisfied 
when $\Phi(x)$ is of the form $-\theta\|x\|^2$ for some positive real 
$\theta>0$ and according to Definition~\ref{def-proper-set}, 
Lemma~\ref{lem-convex-funct}, and the result below.

\begin{lem}\label{Phi=parabola} 
Let $u:U\to[-\infty,M]$ be an \usc\ function bounded from 
above by $M\in\rbb$ on a non-empty open set $U\subset\cbb^n$. 
Define $F_\theta=u_U^\Phi$ for the quadratic function 
$\Phi=-\theta\|{\cdot}\|^2$ and the real parameter $\theta>0$. Given 
any point $y\in{U}$, it lies in the proper set of definition for the 
sup-convolution $F_\theta$, whenever the image $u(y)\in\rbb$ is finite 
and the parameter $\theta\gg1$ is large enough. Moreover, the following 
identity holds for every $y\in{U}$,
\begin{equation}\label{eqn1-Phi=parabola}
u(y)=\lim_{\theta\to\infty}F_\theta(y).
\end{equation} 
\end{lem}

\begin{proof} The result is trivial when $u\equiv-\infty$, because 
$F_\theta=u_U^\Phi$ would be constant and identically equal to 
$-\infty$. Thus, we suppose from now on that it is not the case. A 
direct consequence is deduced from statements (\ref{lem-sup-conv-2}) 
and (\ref{lem-sup-conv-3}) in Lemma~\ref{lem-sup-conv}: the image 
of $F_\theta=u_U^\Phi$ lies inside the real interval $(-\infty,M]$. 
Moreover, given any fixed point $y\in{U}$, the supremum in the 
Definition~\ref{def-sup-conv} of the sup-convolution $F_\theta(y)$ 
only needs to be calculated in the intersection of $U$ with a 
compact ball $\overline{B_\rho}(y)$ of centre at $y$ and radius 
$\rho>0$ large enough; i.e.,
\begin{equation}\label{eqn2-Phi=parabola} 
F_\theta(y)=\sup\big\{u(x)-\theta\|y{-x}\|^2: 
x\in{U\cap}\overline{B_\rho}(y)\big\} 
\end{equation} 
for any $\rho^2>\frac{M-F_\theta(y)}\theta\geq0$, because 
$u(x)-\theta\|y{-x}\|^2<F_\theta(y)$ whenever $x\in{U}$ lies in 
the complement of $\overline{B_\rho}(y)$. Assume for some moments 
that $u(y)\in\rbb$ is finite for $y\in{U}$ fixed, the statement 
(\ref{lem-sup-conv-1}) in Lemma~\ref{lem-sup-conv} implies that the 
sup-convolution $F_\theta(y)$ is bounded from below by $u(y)$ for every 
$\theta>0$. Whence, if we take the parameter $\theta\gg1$ large enough 
in (\ref{eqn2-Phi=parabola}), we can chose a radius $\rho>0$ small 
enough such that the compact ball $\overline{B_\rho}(y)$ is contained 
in $U$. Since $u$ is \usc, the maximum in (\ref{eqn-sup-conv}) 
and (\ref{eqn2-Phi=parabola}) is attained at some point in 
$\overline{B_\rho}(y)$, and so $y$ lies in the proper set of definition 
for $F_\theta=u_U^\Phi$ according to Definition~\ref{def-proper-set}.

Finally, we show (\ref{eqn1-Phi=parabola}) by considering two cases. 
If $u(y)\in\rbb$ is finite for $y\in{U}$ fixed, we have already seen 
that $F_\theta(y)$ is bounded from below by $u(y)$ for every $\theta>0$ 
according to point~(\ref{lem-sup-conv-1}) in Lemma~\ref{lem-sup-conv}. 
Hence, $\frac{M-F_\theta(y)}\theta$ converges to zero when $\theta$ 
goes to $\infty$, and so we can apply the limit when $\rho\to0$ into 
(\ref{eqn2-Phi=parabola}) in order to deduce (\ref{eqn1-Phi=parabola}). 
Since $u$ is \usc, we easily have that 
\begin{equation}\label{eqn3-Phi=parabola} 
u(y)\leq\lim_{\theta\to\infty}F_\theta(y)\leq\lim_{\rho\to0}
\sup\big\{u(x):x\in{U\cap}\overline{B_\rho}(y)\big\}\leq{u}(y).
\end{equation} 

We proceed by contradiction when $u(y)=-\infty$ for the point 
$y\in{U}$. Assume that $F_\theta(y)$ in (\ref{eqn1-Phi=parabola}) 
does not converge to $-\infty$ when $\theta\to\infty$. It is easy to 
deduce from (\ref{eqn-sup-conv}) or (\ref{eqn2-Phi=parabola}) that the 
sup-convolution $F_\theta(y)$ is monotonically decreasing with respect 
to $\theta$, and so $F_\theta(y)$ is bounded from below by some 
$\beta\in\rbb$ for every $\theta>0$. We may now proceed as in the 
previous paragraphs to deduce that equation (\ref{eqn3-Phi=parabola}) 
also holds. This result yields a contradiction, because $u(y)$ is equal 
to $-\infty$, and so both terms of equation (\ref{eqn1-Phi=parabola}) 
are equal to $-\infty$ as well. 
\end{proof}

\section{Sup-convolution and \qpsh\ functions} 

As we mentioned in the introduction and the previous section, one of 
the most important results in the theory of viscosity solutions is 
the so called \textit{magic property}: the sup-convolution $u^\Phi$ 
of a viscosity subsolution $u$ is again a viscosity subsolution (on 
the interior of the proper set of definition for $u^\Phi$). Thus, we 
conclude this work by showing below an extension to this magic property. 
As a pair of applications we also deduce that the sup-convolution of a 
\sqpsh\ function is \sqpsh\ and we find a new characterisation for the 
open \qpsc\ subsets $U\subset\cbb^n$. This latter result extends a list 
of characterisations and properties of \qpsc\ sets we collected in 
\cite{TPESZ} and \cite{TPESZ2}.

\begin{prop}\label{prop-sup-conv} 
Let $\Phi:\cbb^n\to[-\infty,\infty)$ be \usc. Given a 
non-empty open set $U\subset\cbb^n$ and a fixed Hermitian matrix 
$\acal\in\cbb^{n\times{n}}$, take any function $g\in\ccal^2(U)$ whose 
complex Hessian $H^{\cbb}g\geq\acal$ is uniformly bounded from below on 
$U$; i.e., the difference $H^{\cbb}g{-\acal}$ is positive semidefinite 
on $U$. Moreover, given any function $u\in\PSH_q(U)$ for a fixed integer 
$q\geq0$, suppose the proper set of definition for the sup-convolution 
$[u{+g}]_U^\Phi$ has non-empty interior $W\neq\emptyset$. If the 
restriction of $[u{+g}]_U^\Phi$ to $W$ is \usc, then 
$[u{+g}]_U^\Phi{-h}$ lies in $\PSH_q(W)$ for every $h\in\ccal^2(W)$ with 
complex Hessian $H^{\cbb}h\leq\acal$ uniformly bounded from above on $W$. 
\end{prop}

Recall the L\"owner partial order ``$\geq$" between Hermitian matrices 
introduced in Lemma~\ref{elliptic-degenerate}.

\begin{proof} The result is trivial when $q\geq{n}$ or 
$[u{+g}]_U^\Phi\equiv-\infty$ on $W$. Thus, we assume from now on that 
$q<n$ and neither $\Phi$ nor $u$ is identically equal to $-\infty$ on 
their respective set of definition. Recall that 
$[u{+g}]_U^\Phi\equiv-\infty$ whenever $\Phi$ or $u$ is identically 
equal to $-\infty$ because of Definition \ref{def-sup-conv} or the 
point~(\ref{lem-sup-conv-3}) in Lemma~\ref{lem-sup-conv}. On the other 
hand, since $W$ is the interior of the proper set of definition for 
the sup-convolution $[u{+g}]_U^\Phi$, its image 
$[u{+g}]_U^\Phi(W)\subset\rbb$ does not contain the extremal values 
$\pm\infty$ according to the point~(\ref{lem-sup-conv-4}) in the same 
Lemma~\ref{lem-sup-conv}.

Define now $F:=[u{+g}]_U^\Phi{-h}$ and consider any 
$\varphi\in\ccal^2(W)$ such that $F-\varphi$ takes its maximum at some 
$\hat{p}\in{W}$. We shall prove that $\ominus(H^{\cbb}\varphi)\leq{q}$ 
at the point $\hat{p}$, so that Definition~\ref{def-visc-qpsh} and 
Theorem~\ref{thm-usc-visc-qpsh} will imply that $F$ is \qpsh\ in the 
classical and viscosity senses on $W$. Since $F-\varphi$ takes its 
maximum at $\hat{p}\in{W}$, formula (\ref{eqn-sup-conv}) yields the 
identities below for all $w\in{W}$ and $y\in{U}$,
\begin{equation}\label{eqn1-prop-sup-conv} 
\begin{array}{r}[u{+}g]_U^\Phi(\hat{p})-(h{+}\varphi) 
(\hat{p})\;\geq\;[u{+}g]_U^\Phi(w)-(h{+}\varphi)(w)\\[2pt]
\geq\;(u{+}g)(y)+\Phi(w{-y})-(h{+}\varphi)(w). 
\end{array}\end{equation}

Since $\hat{p}$ lies in the interior $W$ of the proper set of 
definition for $[u{+g}]_U^\Phi$, Defi\-ni\-tion~\ref{def-proper-set} 
implies the existence of a point $\hat{x}\in U$, such that
\begin{equation}\label{eqn2-prop-sup-conv}
(u{+}g)(\hat{x})+\Phi(\hat{p}{-}\hat{x})
=[u{+}g]_U^\Phi(\hat{p})\in\rbb.
\end{equation}

Now fix $w=f(y):=y-\hat{x}+\hat{p}$, so that $w-y$ is equal to 
$\hat{p}-\hat{x}$. There is then a small enough neighbourhood 
$\Omega$ of $\hat{x}$ in $U$, such that $f(\Omega)\subset{W}$ 
is also a neighbourhood of $\hat{p}=f(\hat{x})$ in $W$. Equations 
(\ref{eqn1-prop-sup-conv}) and (\ref{eqn2-prop-sup-conv}) yield 
the following inequality for every point $y\in\Omega$,
$$(u{+}g)(\hat{x})-(h{+}\varphi){\circ}f(\hat{x}) 
\geq(u{+}g)(y)-(h{+}\varphi){\circ}f(y).$$

Whence, if we define $\Upsilon:=-g+(h{+\varphi)\circ}f$, the restriction 
of $u-\Upsilon$ to $\Omega$ attains its maximum at $\hat{x}\in{\Omega}$. 
Since $u(\hat{x})\in\rbb$ is finite according to 
(\ref{eqn2-prop-sup-conv}), Lemma~\ref{lem-extension} implies the 
existence of a test function $\psi\in\ccal^2(U)$ such that $u-\psi$ 
has a strict global maximum at $\hat{x}\in{U}$ and the complex Hessians 
$H^{\cbb}\Upsilon(\hat{x})$ and $H^{\cbb}\psi(\hat{x})$ coincide. The 
given hypotheses also imply that 
$H^{\cbb}h(\hat{p})\leq{H^\cbb}g(\hat{x})$, because $H^{\cbb}h\leq\acal$ 
on $W$ and $\acal\leq{H^\cbb}g$ on $U$; see for example section~6.2 of 
\cite{Zhang}. Since $f$ is a translation, we can use all the previous 
results to deduce that,
\begin{equation}\label{eqn3-prop-sup-conv} 
H^{\cbb}\psi(\hat{x})=H^{\cbb}\Upsilon(\hat{x})=H^{\cbb}({h+} 
\varphi)(\hat{p})-H^{\cbb}g(\hat{x})\leq{H^\cbb}\varphi(\hat{p}). 
\end{equation}

Now then, the operator $\ominus(H^{\cbb}\psi)\leq{q}$ at the 
point $\hat{x}$, because $u-\psi$ has a strict global maximum at 
$\hat{x}\in{U}$ and $u\in\PSH_q(U)$. Recall the given hypotheses, 
Definition~\ref{def-visc-qpsh}, and Theorem~\ref{thm-usc-visc-qpsh}. 
Moreover, $\ominus(H^{\cbb}\varphi)\leq{q}$ at $\hat{p}\in{W}$ because 
of (\ref{eqn3-prop-sup-conv}) and Lemma~\ref{elliptic-degenerate}, so 
that the restriction of $[u{+g}]_U^\Phi{-h}$ to $W$ is \qpsh\ in the 
classical and viscosity senses on $W$.
\end{proof}

We can obviously define $g(y)=h(y)=y^t\!\acal\overline{y}$ 
in the previous proposition for any fixed Hermitian matrix 
$A\in\cbb^{n\times{n}}$, so that the complex Hessians $H^{\cbb}g$ 
and $H^{\cbb}h$ given in (\ref{complex-hessian}) are constant and 
equal to $\acal$. The following results are a direct consequence of 
Lemma~\ref{lem-convex-funct} and Proposition~\ref{prop-sup-conv}.

\begin{cor}\label{cor-sup-conv}
Let $\Phi\in\ccal(\cbb^n)$ be semiconvex with constant $\delta>0$, 
and $u$ be \qpsh\ on a non-empty open set $U\subset\cbb^n$ and 
for an integer $q\geq0$. Given a constant Hermitian matrix 
$\acal\in\cbb^{n\times{n}}$, define $g(y)=y^t\!\acal\overline{y}$ 
and suppose that the proper set of definition for the sup-convolution 
$[u{+g}]_U^\Phi$ has non-empty interior $W\neq\emptyset$. The 
restriction of $[u{+g}]_U^\Phi{-g}$ to $W$ is then continuous 
and \qpsh\ on $W$. 
\end{cor}

Observe that $\Phi\not\equiv-\infty$ because of the given hypotheses. 
Moreover, one can fix $g\equiv0$ in Corollary~\ref{cor-sup-conv}, so 
that the sup-convolution $u^\Phi$ of a \qpsh\ function $u$ is again 
\qpsh. We can show a similar result for \sqpsh\ functions. Recall the 
following definition from \cite{Bu,Va,Dieu}.

\begin{defn}\label{defn-strict-qpsh}
Let $u:U\to[-\infty,\infty)$ be \usc\ on a non-empty open set 
$U\subset\cbb^n$. One says that $u$ is \sqpsh\ on $U$ if and only if for 
every point $y\in{U}$ there are a real $\varepsilon>0$ and a non-empty 
open ball $B\subset{U}$ with centre at $y$, such that the restriction 
of $u{-}\varepsilon\|{\cdot}\|^2$ to $B$ is \qpsh\ on $B$. 
\end{defn}

Notice that $u\equiv-\infty$ is \sqpsh; and it is the only constant 
function with this property. One must also say that the \textit{\sqpsh} 
functions are sometimes called \textit{strongly \qpsh}.

\begin{cor}
Let $u:U\to[-\infty,\infty)$ be strictly \qpsh\ on a non-empty open 
set $U\subset\cbb^n$ and for an integer $q\geq0$. Given any compactly 
contained open set $\Omega\Subset{U}$, assume the proper set of 
definition for the sup-convolution $[u|_\Omega]_{\Omega}^\Phi$ has 
a non-empty interior $W\neq\emptyset$, where $\Phi\in\ccal(\cbb^n)$ 
is semi\-con\-vex with constant $\delta>0$. The restriction of 
$[u|_\Omega]_{\Omega}^\Phi$ to $W$ is then continuous and 
strictly \qpsh\ on $W$. 
\end{cor}

\begin{proof} We assert that the restriction of $u{-}\delta\|{\cdot}\|$ 
to $\Omega$ is \qpsh\ on $\Omega$ for some small real $\delta>0$. Since 
$\overline{\Omega}\subset{U}$ is a compact set, 
Definition~\ref{defn-strict-qpsh} yields the existence of $m$ 
positive real numbers $\{\varepsilon_j\}_{j=1}^m$ and a finite cover 
$\{B_j\}_{j=1}^m$ of $\overline{\Omega}$ by non-empty open balls, such 
that the restriction of ${u-}\varepsilon_j\|{\cdot}\|^2$ to $B_j$ is 
\qpsh\ on $B_j$ for every $j=1,...,m$. The real $\delta>0$ we are 
looking for is the minimum of $\{\varepsilon_j\}_{j=1}^m$.

Corollary~\ref{cor-sup-conv} now implies that the restriction 
of $[u|_\Omega]_{\Omega}^\Phi{-}\varepsilon\|{\cdot}\|$ to $W$ 
is also continuous and \qpsh\ on $W$, because we only need to fix 
$g=\varepsilon\|{\cdot}\|$. Hence, $[u|_\Omega]_{\Omega}^\Phi$ 
is strictly \qpsh\ on $W$.
\end{proof}

Another interesting result is that the complex Hessian $H^{\cbb}u$ of 
a strictly \qpsh\ function $u$ has at least $n{-q}$ strictly positive 
eigenvalues in the viscous sense. Recall the operator $\oplus$ defined 
in Lemma~\ref{elliptic-degenerate}.

\begin{lem}
Let $u:U\to[-\infty,\infty)$ be strictly \qpsh\ on an open 
set $U\subset\cbb^n$ and for an integer $q\geq0$. Assume that 
$u\not\equiv-\infty$. Given any test function $\varphi\in\ccal^2(U)$, 
the operator $\oplus(H^{\cbb}\varphi)\geq{n-}q$ at those points 
$\hat{p}\in{U}$ where $u-\varphi$ attains its maximum; i.e., the 
complex Hessian $H^{\cbb}\varphi$ in (\ref{complex-hessian}) has at 
least $n{-q}$ strictly positive eigenvalues at those points $\hat{p}$. 
\end{lem}

\begin{proof} The result is trivial when $q\geq{n}$, so we suppose it 
is not the case. According to Definition~\ref{defn-strict-qpsh} and 
Theorem~\ref{thm-usc-visc-qpsh}, there are a real $\varepsilon>0$ and 
a non-empty open ball $B\subset{U}$ with centre at $y$, such that the 
restriction of $u{-}\varepsilon\|{\cdot}\|^2$ to $B$ is \qpsh\ in the 
classical and viscosity senses on $B$. Since the restriction of 
$u-\varphi$ to $\Omega$ also takes its maximum at $\hat{p}\in\Omega$, 
we can add and subtract $\varepsilon\|{\cdot}\|^2$ to $u-\varphi$ and 
consider the new test function $\varphi{-}\varepsilon\|{\cdot}\|^2$, 
so that Definition~\ref{def-visc-qpsh} yields,
$$q\geq\ominus\big(H^{\cbb}(\varphi-\varepsilon
\|{\cdot}\|^2)(\hat{p})\big)=\ominus\big(H^{\cbb}
\varphi(\hat{p})-\varepsilon\mathbbm{1}\big),$$
where $\mathbbm{1}$ is the identity $[n{\times}n]$-matrix. Hence, 
$H^{\cbb}\varphi(\hat{p})-\varepsilon\mathbbm{1}$ has at most $q$ 
strictly negative eigenvalues, and so $H^{\cbb}\varphi(\hat{p})$ has at 
least $n{-q}$ strictly positive eigenvalues, because $\varepsilon>0$. 
\end{proof}

Unfortunately, the converse of the previous lemma does not hold in 
general.

\begin{ex}\label{counter-sqpsh}
The function $\Im^4(z)=\frac{(z-\overline{z})^4}{16}$ obviously lies in 
$\PSH_0(\cbb)$, but it is only strictly \psh\ on $\cbb\setminus\rbb$, 
because its complex Hessian is $3\Im^2(z)$ and it vanishes on the real 
axis $\rbb\subset\cbb$. In the same way, the point~(\ref{qpsh-fcts}) in 
Proposition~\ref{propqpsh} implies that
\begin{equation}\label{eqn-ex-strict-0psh}
f(z):=\Im^4(z)+|\Im(z)|=\Im^4(z)+\max\{\Im(z),-\Im(z)\}
\end{equation}
lies in $\PSH_0(\cbb)$. We easily have $H^{\cbb}f(z)=3\Im^2(z)$ 
on $\cbb\setminus\rbb$, so that $f$ is also strictly \psh\ on 
$\cbb\setminus\rbb$. Nevertheless, $f$ is not strictly \psh\ on 
$\rbb$, because for every $\varepsilon>0$ the complex Hessian of 
$f{-\varepsilon\|\cdot}\|^2$ converges to $-\varepsilon$ when 
$z\not\in\rbb$ moves close enough to $\rbb$, so it becomes 
strictly negative near $\rbb$.

On the other hand, let $\varphi\in\ccal^2(\cbb)$ be any test function, 
such that $f-\varphi$ takes its maximum at some $\hat{p}\in\cbb$. 
Notice that $f-\varphi$ cannot take its maximum in the real axis 
$\rbb\subset\cbb$ because of the absolute value $|\Im(z)|$ in 
(\ref{eqn-ex-strict-0psh}), so that $\hat{p}\not\in\rbb$. Since 
$f-\varphi$ is $\ccal^2$-smooth on $\cbb\setminus\rbb$, we can 
repeat word by word the analysis done in the second part of the 
proof of Lemma~\ref{lem-smooth-visc-qpsh} to deduce that
$$H^{\cbb}\varphi(\hat{p})\geq{H^\cbb}f(\hat{p})=3\Im^2(\hat{p})>0$$ 
at every point $\hat{p}$ where $f-\varphi$ takes its maximum.
\end{ex}

Another interesting example appears after considering 
the result presented in the point~(\ref{lem-sup-conv-7}) of 
Lemma~\ref{lem-sup-conv}: The sup-convolution $\chi^\Phi$ is equal to 
the composition of the distance $\mathrm{dist}_X$ with a monotonically 
decreasing function $f$. Whence, if the characteristic function $\chi$ 
in (\ref{charac-funct-1}) or (\ref{charac-funct-2}) below is \qpsh, 
Lemma~\ref{lem-sup-conv} and Proposition \ref{prop-sup-conv} would imply 
that $\chi^\Phi$ and $ f\circ\mathrm{dist}_X$ are both \qpsh\ as well. 
In this case, S{\l}odkowski characterised in \cite{Sl} those closed 
sets $X\subset\cbb^n$ for which its characteristic function 
$\chi\in\PSH_q(\cbb^n)$. S{\l}odkowski actually produced several 
characterisations in \cite{Sl}, one of them says that 
$\chi\in\PSH_q(\cbb^n)$ if and only if the open complement 
$\cbb^n{\setminus}X$ is $(q{-1})$-\psc; see for example Theorems~2.5 
and~4.2 in \cite{Sl} or Theorem~3.8 in \cite{TPESZ2}. We can restate 
the previous analysis as follows.

\begin{cor}\label{ex-charac-funct}
Let $f:\rbb\to[-\infty,\infty)$ be \usc\ and monotonically decreasing, 
such that $f(0)=0$ and $f(t)<0$ for all $t>0$. Take the Euclidean 
distance $\mathrm{dist}_X$ to a non-empty closed set $X\subset\cbb^n$. 
Given any integer $q\geq0$, the statements below are all equivalent: 
\begin{enumerate}

\item\label{ex-charac-funct-1} 
The characteristic function below lies in $\PSH_q(\cbb^n)$,
\begin{equation}\label{charac-funct-2}
\chi(y):=\left\{\begin{array}{cl}0&\hbox{if~}y\in{X};\\
-\infty&\hbox{otherwise}.\end{array}\right.
\quad\hbox{for}\quad{y}\in\cbb^n.
\end{equation}

\item\label{ex-charac-funct-2} 
The composition $f\circ\mathrm{dist}_X$ lies in $\PSH_q(\cbb^n)$.

\item\label{ex-charac-funct-3} The open set $\cbb^n{\setminus}X$ is 
$(q{-1})$-\psc\ on $\cbb^n$ under the extra condition that $1\leq{q}<n$, 
and so $-\ln\mathrm{dist}_X$ lies in $\PSH_{q-1}(\cbb^n{\setminus}X)$. 
\end{enumerate} 
\end{cor}

Notice that $\chi$ in (\ref{charac-funct-2}) is \usc\ when 
$X\subset\cbb^n$ is closed. Moreover, if we set $q=1$ in the above 
statements, we obtain a new characterisation for domains of 
holomorphy (0-\psh\ domains) $D\subset\cbb^n$: The composition 
$f\circ\mathrm{dist}_X$ is 1-\psh\ for the distance $\mathrm{dist}_X$ 
to the complement $\cbb^n{\setminus}D$ and a monotonically decreasing 
function $f$ as above.

\begin{proof} (\ref{ex-charac-funct-1}) implies 
(\ref{ex-charac-funct-2}). We have $\chi^\Phi=f\circ\mathrm{dist}_X$ 
for $\Phi=f(\|{\cdot}\|)$ according to the point~(\ref{lem-sup-conv-7}) 
in Lemma~\ref{lem-sup-conv} and identity (\ref{eqn2-lem-sup-conv}). 
Moreover, $\cbb^n$ is the proper set of definition for $\chi^\Phi$, 
because $X$ is closed, so that for each $y\in\cbb^n$ there is 
$\hat{x}\in{X}$ with $\mathrm{dist}_X(y)$ equal to $\|\hat{x}{-y}\|$. Since 
$f\circ\mathrm{dist}_X$ is \usc\ and $\chi$ lies in $\PSH_q(\cbb^n)$ by 
hypothesis, Proposition~\ref{prop-sup-conv} imply that 
$\chi^\Phi=f\circ\mathrm{dist}_X$ lies in $\PSH_q$ as well.

(\ref{ex-charac-funct-2}) implies (\ref{ex-charac-funct-1}). Notice that 
$\chi=\lim_{k\to\infty}u_k$ for $u_k:=k\,(f \circ\mathrm{dist}_X)$ and 
every integer number $k\in\nbb$. Since $u_1\geq{u_2}\geq...$, points 
(\ref{qpsh-fcts}) and (\ref{qpsh-decreasing}) in 
Proposition~\ref{propqpsh} implies that $u_k$ and 
$\chi$ both lie in $\PSH_q(\cbb^n)$.

(\ref{ex-charac-funct-1}) is equivalent to (\ref{ex-charac-funct-3}). 
Theorems~2.5 and~4.2 in \cite{Sl} state that $\chi$ lies in 
$\PSH_q(\cbb^n)$ if and only if the complement $\cbb^n{\setminus}X$ 
is $(q{-1})$-\psc. Finally, the equivalence with the property that 
$-\ln\mathrm{dist}_X$ lies in $\PSH_{q-1}(\cbb^n{\setminus}X)$ follows 
from Definition~3.3 in \cite{TPESZ} or Proposition~3.3 in \cite{TPESZ2}. 
\end{proof}

This result extends a list of characterisations and properties of \qpsc\ 
sets we collected in \cite{TPESZ} and \cite{TPESZ2}. Besiedes, we must 
point out that the previous results can be extended to consider locally 
closed sets $X\subset\cbb^n$ instead of closed ones. However, we must be 
careful, because the proper set of definition for $\chi^\Phi$ may not be 
the whole space $\cbb^n$. Indeed, for each $y\in\cbb^n$ there is 
$\hat{x}\in\overline{X}$ with $\mathrm{dist}_X(y)$ equal to 
$\|\hat{x}{-y}\|$, but $\hat{x}$ may not lie in $X$, when 
it is different from $\overline{X}$.

\addcontentsline{toc}{section}{References}

\bibliographystyle{siam}
\bibliography{qpsc3}

\end{document}